\documentclass[10pt,twoside]{article}
\usepackage{amssymb}

\usepackage{amsmath}
\usepackage{latexsym}

\newcommand{\COM}[1]{}

\newtheorem{theorem}{Theorem}[section]

\newtheorem{lemma}{Lemma}[section]
\newtheorem{proposition}{Proposition}[section]

\newtheorem{remark}{\normalfont\scshape Remark}[section]

\newenvironment{proof}{\noindent\textsc{Proof.\/}}{}

\newcommand{\subj}[2]{\textsf{AMS 2000 subject classifications.}
Primary {#1}; Secondary {#2}.\newline}
\newcommand{\key}[1]{\textsf{Keywords and phrases.} {#1}.\newline}
\newcommand{\abb}[1]{\textsf{Abbreviated title.} {#1}.}

\newcommand{\fot}[5]{\renewcommand\thefootnote{}
\footnotetext{\parindent=0.0mm \vskip-3mm \subj{#1}{#2}\key{#3}\abb{#4}
\newline\textsf{Date.} \date{\today}}}

\newcommand{\erw}{elephant random walk}
\newcommand{\cfn}{{\cal F}_n}

\newcommand{\cgn}{{\cal G}_n}

\newcommand{\mem}{{\mathfrak M}}
\newcommand{\xnp}{X_{n+1}}
\newcommand{\tnp}{T_{n+1}}
\newcommand{\tnz}{T_n^2}
\newcommand{\tnpz}{T_{n+1}^2}
\newcommand{\exnpf}{E(X_{n+1}\mid \cfn)}

\def\vsb{\hfill$\Box$}

\def\vsp{\vskip-8mm\hfill$\Box$\vskip3mm}

\newcommand{\be}{\begin{equation}}
\newcommand{\ee}{\end{equation}}
\newcommand{\bea}{\begin{eqnarray}}
\newcommand{\eea}{\end{eqnarray}}
\newcommand{\beaa}{\begin{eqnarray*}}
\newcommand{\eeaa}{\end{eqnarray*}}
\newcommand{\beal}{\begin{aligned}}
\newcommand{\eeal}{\end{aligned}}

\newcommand{\var}{\mathrm{Var\,}}

\newcommand{\sumk}{\sum^n_{k=1}}
\newcommand{\sumin}{\sum_{n=1}^\infty}

\newcommand{\ttt}[1]{\quad\mbox{ #1}\quad}

\newcommand{\asto}{\stackrel{a.s.}{\to}}
\newcommand{\pto}{\stackrel{p}{\to}}

\newcommand{\dto}{\stackrel{d}{\to}}

\newcommand{\nifi}{n\to\infty}

\begin{document}
\date{}
\title{\textsf{Elephant random walks with delays}}
\author{Allan Gut\\Uppsala University \and Ulrich Stadtm\"uller\\
Ulm University}
\maketitle

\begin{abstract}\noindent
In the simple random walk the steps are independent, viz., the walker has no memory. In contrast, in the \erw\ (ERW), which was introduced by Sch\"utz and Trimper \cite{erwdef} in 2004,  the walker remembers the whole past and the next step always depends on the whole path so far. One extension, as suggested in a recent paper by Bercu et al.\ \cite{bercu2}, is to allow for delays, that is, to put mass at zero. Our  aim is  to extend known result for the ordinary ERW to  \emph{\erw s with delays\/} (ERWD).
\end{abstract}

\fot{60F05, 60G50,}{60F15, 60J10}{Elephant random walk, delay, law of large numbers, asymptotic (non)normality, method of moments, difference equation, Markov chain}
{ERW with delays}

\section{Introduction}
\setcounter{equation}{0}
\markboth{A.\ Gut and U.\ Stadtm\"uller}{Elephant random walk}

In the classical \emph{simple\/} random walk the steps are equal to plus or minus one and independent---$P(X=1)=1-P(X=-1)=p$, ($0<p<1$); the walker has no memory. This random walk is, in particular, Markovian. Motivated by applications, although interesting in its own right, is the case when the walker has some memory. The so called \erw\ (ERW),  for which ''the next step'' depends on the whole process so far, was introduced by Sch\"utz and Trimper \cite{erwdef} in 2004, the name being inspired by the fact that elephants have a very long memory.

The first, more substantial (theoretical), papers on ERWs are, to the best of our knowledge, the paper by Coletti et al., \cite{Coletti}, and Bercu, \cite{bercu}. Our predecessor, \cite{111}, is devoted to the situation when the elephant has only a limited memory, more precisely, to the case when he or she remembers only some distant past, only a recent past, or a mixture of both. These models
behave very differently mathematically in that some of the walks are still non-Markovian others are Markovian, but there is no convenient martingale around (as in \cite{bercu}). Moreover we do not encounter any phase transitions (as, e.g., in \cite{bercu}).

In the present sequel we introduce the possibility of delays in that the elephant, in addition, has a choice of staying put in every step. After having defined the various models in Section \ref{defs}, and preliminaries in Section \ref{aux},  some results  for general \emph{\erw s with delays\/} (ERWD) are obtained in Section \ref{erwdel}.  In Sections \ref{anfang1} and \ref{anfang2} the elephant remembers the distant  past, and in Section \ref{ende1} the recent past. It is quite common that generalizations of results are proved via extensions of corresponding earlier proofs. In Section \ref{ende1}, however, the scenery differs, as we shall see, drastically from the parent one in \cite{111}. We close with Section \ref{beides},  where we consider a mixed case, more precisely,  when the memory consists of the first step \emph{and\/} the most recent step, followed by some final remarks.

\section{Background}
\setcounter{equation}{0}\label{defs}
The \erw\ is defined as a simple random walk, where, however, the steps are not i.i.d.\ but dependent as follows. The first step $X_1$ equals 1 with probability $s\in [0,1]$ and is equal to $-1$ with probability $1-s$. After $n$ steps, at position $S_n=\sumk X_k$, one defines
\bea\label{20}\xnp=\begin{cases} +X_{K},\ttt{with probability} p\in[0,1], \\-X_{K},\ttt{with probability} 1-p,\end{cases}\eea
where $K$ has a uniform distribution on the integers $1,2,\ldots,n$. With $\cgn=\sigma\{X_1,X_2,\ldots,X_n\}$ this means (formula (2.2) of \cite{bercu}) that  
\bea
E(\xnp\mid\cgn)=(2p-1)\cdot\frac{S_n}{n},\label{21}\eea
 after which, setting  $a_n=\Gamma(n)\cdot\Gamma(2p)/\Gamma(n+2p-1)$, it turns out that
$\{M_n=a_nS_n,\,n\geq 1\}$ is a martingale.
In Section 5 of a  follow-up, \cite{bercu2}, a variation of the model is introduced, allowing for a third possibility, $X_{n+1}=0$.

Our aim is  to extend the results of Bercu \cite{bercu} and of our paper \cite{111} to this three-point case. We first consider the extension of (\ref{20}) to the case
\bea\label{23}\xnp=\begin{cases} +X_{K},\ttt{with probability} p\in[0,1], \\-X_{K},\ttt{with probability} q\in[0,1],\\\phantom{+)}0,\ttt{\,\,  with probability}r\in[0,1],\end{cases}\eea
where  $p+q+r=1$, and where $K$ has a uniform distribution on the integers $1,2,\ldots,n$.  Everything reduces, of course, to \cite{111} if $r=0$.

Next we assume, as in \cite{111}, that the elephant has only a restricted memory by first  considering the case when the elephant remembers  the first step only. In that case,
\[\xnp=\begin{cases} +X_{1},\ttt{with probability} p\in[0,1], \\-X_{1},\ttt{with probability} q\in[0,1], \\\phantom{+)}0,\ttt{ with probability} r\in[0,1].\end{cases}\]
This is the topic of Section \ref{anfang1}. We begin by assuming that $X_1=1$, and generalize our findings in this setting (for simplicity) to the case  $s=p$. Partial sums are denoted by $T_n$, $n\geq1$, when the first variable(s) is/are fixed and  $S_n$ when they are random.

In order to move from $T_n$ to $S_n$ we also need to discuss the behavior of the walk when the initial value equals $-1$ or $0$. In the former case the evolution of the walk is the same \emph{except\/} for the fact that the trend of the walk is reversed, viz., the corresponding walk equals the mirrored image in the time axis. This implies that the mean after $n$ steps equals $-E(T_n)$, but the dynamics being the same implies that the variance remains the same.  When the initial value equals zero the process is a zero-process.

Somewhat more sophisticated is when the memory covers the first two steps (Section \ref{anfang2}).
Technically different is when the elephant only remembers the most recent past, in particular only the last step, in which case
\[\xnp=\begin{cases} +X_{n},\ttt{with probability} p\in[0,1], \\-X_{n},\ttt{with probability}q \in[0,1], \\\phantom{+)}0,\ttt{\,  with probability}r\in[0,1].\end{cases}\]
Finally, in order to avoid special effects we assume throughout that $0<p,q,r<1$. We use the standard $\delta_a(x)$ to  denote the distribution function with a jump of height one at $a$.
Constants $c$ and $C$ are always numerical constants that may change between appearances.

\section{Some auxiliary material}
\setcounter{equation}{0}\label{aux}

For easier access to the arguments below we collect here some auxiliary material.\smallskip

\noindent\textbf{(i)}\quad The following (well-known) result (which is a special case of the Cram\'er--Slutsky theorem) will be used in order to go from a special case to a more general one.

\begin{proposition}\label{knep} Let $\{U_n,\,n\geq1\}$ be a sequence of random variables, and suppose that $V$ is independent of all of them. If $U_n\dto U$ as $\nifi$, then $U_nV\dto UV$ as $\nifi$.
\end{proposition}

\noindent\textbf{(ii)}\quad Next is a result concerning the case of a restricted memory.
Let $\{S_n,\,n\geq1\}$ be an ERW, let $\{\cfn,\,n\geq1\}$ denote the   $\sigma$-algebras generated by the memory of the elephant and let $\cgn= \sigma\{X_1,X_2,\ldots,X_n\}$ stand for  the full memory. Following is an extension of (\ref{21}).    Let $\mem =$ the memory of the elephant, and set $I_n=\{i\leq n: i \in \mem\}$.  Then,
\bea\label{uli}
E(\xnp\mid \cfn)
=(p-q)\cdot\frac{\sum_{i\in I_n}X_i}{|I_n|}.
\eea
We also need a formula for the case when we condition on steps that are not contained in the memory. In words, if they do not, the elephant does not remember them, and, hence, cannot choose among them in a following step. Technically, let $I\subset \{1,2,\ldots,n\}$ be an arbitrary set of indices, such that $I\cap I_n=\emptyset$. Then
\bea\label{uli2}
E(\xnp\mid \sigma\{I_n\cup I\})=E(\xnp\mid\cfn)=(p-q)\frac{\sum_{i\in I_n}X_i}{|I_n|}.\eea
This will be useful several times for the computation of second moments:
\bea\label{uli5}
E(S_{n+1}^2)=E(S_n^2)+\frac{2(p-q)}{|I_n|}E\Big(S_n\sum_{i\in I_n}X_i\Big) +E(\xnp^2).
\eea
\textbf{(iii)}\quad Difference equations appear in the course of the proofs. Here are some well-known facts about linear difference equations.
\begin{proposition}\label{diff}
\emph{(i)} Consider the first order equation
\[x_{n+1}=a\, x_n +b_n \ttt{for} n \ge 1, \ttt{with} x^*_1 \ttt{given.}\]
Then
\[x_n=a^{n-1}x^*_1+\, \sum_{\nu=0}^{n-2}a^\nu b_{n-1-\nu}.\]
If, in addition, $|a|<1$ and $b_n=bn^\gamma$ with $\gamma>-1$, then
\[x_n=\frac{b_{n-1}}{1-a} -\frac{\gamma ab_{n-1}}{n(1-a)^2}\big(1+o(1)\big)\ttt{as}\nifi.\]
\emph{(ii)} If, in particular, $|a|<1$ and $x_{n+1}=ax_n+b$, then
\[x_n=\frac{b}{1-a}+a^{n-1}\big(x_1^*-\frac{b}{1-a}\big)=\frac{b}{1-a}\big(1+o(1)\big)\ttt{as}\nifi.\]
\end{proposition}

\section{Elephant random walks with delays}\label{erwdel}
\setcounter{equation}{0}
We thus depart from the model described in (\ref{23}), and set $S_n=\sumk X_k$, $n\geq1$, with $S_0=0$. Since the elephant remembers everything, formula (\ref{uli}) tells us that
\[\exnpf=E(\xnp\mid \cfn)=(p\cdot 1 + q\cdot (-1) +r\cdot0)\frac{S_n}{n} =(p-q)\frac{S_n}{n} ,\]
 so that 
\[E(S_{n+1}\mid \cfn)=S_n+(p-q)\frac{S_n}{n}=\big(1+\frac{p-q}{n}\big)S_n .\]
Setting $\gamma_k=1+\frac{p-q}{k}$, $k\geq1$, and 
$a_n=\prod_{k=1}^{n-1}\gamma_k^{-1}=\frac{\Gamma(p-q+1)\cdot \Gamma(n)}{\Gamma( n+p-q)}$, $n\geq1$, we define
$M_n=a_nS_n$, $n\geq1$, and note that $\{M_n,\,n\geq1\}$ is a martingale (for convenience, see also Problem 10.6 of \cite{g13}). The asymptotics of the martingale determines the asymptotics of the  ERWD.

The next step is to modify \cite{bercu}, formula (A.8). Toward that end we set $\nu_n=\sumk a_k^2$, after which asymptotics of the $\Gamma$-function tells us that
\bea
\nu_n\begin{cases}\,\sim (\Gamma(p-q+1))^2\cdot\dfrac{n^{1-2(p-q)}}{1-2(p-q)},&\ttt{for}p-q<1/2,\\[2mm]
\,\sim \dfrac{\pi}{4}\log n,&\ttt{for}p-q=1/2,\\[2mm]
\,\leq\, C,&\ttt{for}p-q>1/2,\end{cases}\label{nun}
\quad\ttt{as}\nifi,\eea
which determines the diffusive, critical and superdiffusive regimes, respectively. 

From here on the mathematics can be copied and pasted from \cite{bercu} (and obviously modified) plus using the fact that $S_n=0, \; \forall n$ if $X_1=0$. The following result emerges.
\begin{theorem} \label{thm41} \emph{(a)}\quad For $p-q<1/2$,
\[\frac{S_n}{n}\asto0\ttt{as}\nifi\ttt{and}\frac{S_n}{\sqrt{n}}\dto (p+q){\cal N }_{0,\frac{p+q}{1-2(p-q)}}+r \,\delta_0(x)\ttt{as}\nifi.\]
\emph{(b)}\quad For $p-q=1/2$,
\[\frac{S_n}{\sqrt{n}\log n}\asto0\ttt{and}\frac{S_n}{\sqrt{n\log n}}\dto (p+q){\cal N }_{0,p+q}+r\,\delta_0(x)\ttt{as}\nifi.\]
\emph{(c)}\quad For $p-q>1/2$, 
\[\frac{S_n}{n^{p-q}}\asto L\ttt{as}\nifi,\]
where $L$ is a non-degenerate random variable. Moreover,
\[E(L)=\frac{p-q}{\Gamma(1+p-q)}\ttt{ and } E(L^2)=\frac{p+q}{\big(2(p-q)-1\big)\Gamma\big(2(p-q)\big)}.\]
\end{theorem}
\begin{remark}\emph{The first four moments of $L$ are given in \cite{bercu} for the case $r=0$. All moments are given in \cite{bercu2}.}
\end{remark}
\begin{remark}\emph{ For $r=0$ (and, thus, $q=1-p$) we rediscover the results of Bercu \cite{bercu} (with $s=p$). If $q=0$ the steps are either the same as the chosen step or zero. Phase transition occurs at $p=1/2$. If $p=0$ the steps oscillate unless they are zero, and the elephant is superdiffusive.}\vsb\end{remark}

Next we exploit a device from \cite{111} in order to extend Theorem \ref{thm41} to allow for general step sizes. Toward that end, let $\{\widetilde{S}_n,\,n\geq1\}$ be an ERWD, and suppose that $Y$ is a random variable with distribution function $F_Y$ that is independent of the walk. If $\widetilde{S}_n/b_n\asto Z$ as $\nifi$ for some normalizing positive sequence $b_n\to\infty$ as $\nifi$, and some random variable $Z$, it follows trivially that $Y\widetilde{S}_n/b_n\asto YZ$ as $\nifi$. Now consider the ERWD for which $\widetilde{X}_1\equiv 1$, and let the random variables $\widetilde{X}_n$, $n\ge 2$, be constructed as in Section 2 with this special $\widetilde{X}_1$ as starting point.   Furthermore, let $Y$ be a random variable, independent of $\{ \widetilde{X}_n,\,n\geq1\}$,  and consider,
for $n\geq1$, $X_n=Y\cdot \widetilde{X}_n$, and, hence, $S_n= Y\cdot \widetilde{S}_n$. 

The following theorems hold for $S_n=Y\tilde{S}_n$. The proofs are ''the same'' as in \cite{111}, Section 4.
\begin{theorem}\quad
\emph{(a)}\quad For $p-q<1/2$, \quad $\dfrac{S_n}{n}\asto 0\ttt{as}\nifi$;\\[1mm]
\emph{(b)}\quad For $p-q=1/2$, \quad $\dfrac{S_n}{\sqrt{n}\log n}\asto 0\ttt{as}\nifi$;\\[1mm]
\emph{(c)}\quad For $p-q>1/2$, \quad $\dfrac{S_n}{n^{p-q}}\asto YL\ttt{as}\nifi$,\\[1mm]
where $L$ is a non-degenerate random variable.
\end{theorem}
Next we consider convergence in distribution.  
\begin{theorem}\label{thmgerw}
For $p-q<1/2$,
\[ \frac{S_n}{\sqrt{n}} \dto (p+q)\int_{\mathbb{R}\backslash\{0\}}  {\cal N }_{0,\frac{1}{1-2(p-q)}} (\cdot/|t|) \,dF_Y(t) +\big((p+q)P(Y=0)+r\big)\cdot\delta_{0}(\cdot)\ttt{as}\nifi.\]
 Moreover, if $E(Y^2)<\infty$, then $E(S_n/\sqrt{n})\to 0 $ and $E((S_n/\sqrt{n})^2)\to E(Y^2)/(1-2(p-q))$ as $\nifi$.
\end{theorem}
\begin{remark}\emph{For the critical case one similarly obtains 
\[ \frac{S_n}{\sqrt{n\log n}} \dto (p+q)\int_{\mathbb{Y}\backslash\{0\}}  {\cal N }_{0,p+q} (\cdot/|t|) \,dF_Y(t) +\big((p+q)P(Y=0)+r\big)\cdot\delta_{0}(\cdot)\ttt{as}\nifi.\]
The supercritical case has a different evolution and no analogous result exists.}\vsb\end{remark}

\section{Remembering only the distant past 1}\label{anfang1}
\setcounter{equation}{0}
Suppose that the elephant only remembers the first step, i.e., that $\cfn=\sigma\{X_1\}$, and, initially, that $X_1=1$ (recall that partial sums are then denoted with the letter $T$). Then, for all $n\geq 1$,
\[\exnpf=E(\xnp\mid X_1)=p\cdot 1 + q\cdot (-1) +r\cdot 0=p-q=E(\xnp),\]
 and, hence, 
\[E(T_{n+1})= 1+n\,(p-q).\]
Moreover, applying (\ref{uli5}) to $T_n$, we find that
\beaa
E(\tnpz)&=& E(T_n^2)+2\,(p-q)E(T_n)+p+q\\
&=&E(\tnz)+2(p-q)\big(1+(n-1)(p-q)\big)+p+q\\
&=&E(\tnz)+2(p-q)^2n +2(p-q)(1-p+q)+p+q,
\eeaa
which, after telescoping, yields
\[E(\tnpz)=1+(p-q)^2n(n+1)+\big(2(p-q)(1-p+q)+p+q\big)n,\]
and, finally, 
\[\var (\tnp) =n\big(p+q-(p-q)^2\big).\]
We note that mean and variance coincide with those of a a \emph{delayed} simple random walk, except for the fact that the first step is always equal to one. As in our predecessor, \cite{111}, one can, in fact, prove that this is, indeed, the case.  
This permits us to draw the following conclusion.
\begin{proposition}\label{prop51}  The strong law of large numbers, the central limit theorem, and the law of the iterated logarithm all hold for  $\{T_n,\,n\geq1\}$.
\end{proposition}
 If, on the other hand, the first step is equal to $-1$, then, by symmetry,
$E(\tnp)=-n(p-q)-1$, the variance remains the same (recall the discussion toward the end of Section \ref{defs}), and
Proposition \ref{prop51} applies. If $X_1=0$ everything is trivial.  This implies, for example,  the following strong law:
\bea
\frac{S_n}{n}&\asto& (p-q)\cdot I\{X_1=1\}-(p-q)\cdot I\{X_1=-1\}+0\cdot I\{X_1=0\}\nonumber\\
&=&(p-q)\cdot\textrm{sign}(X_1)\ttt{as}\nifi.\label{LLN}
\eea
As for distributional convergence, we are (asymptotically) confronted with two normal distributions and the $\delta_0$-distribution. Moreover, if $X_1=\pm1$,  then $\var( S_n) \asymp n^2$  as $\nifi$, and an ordinary CLT is not valid.

However, the following limit result is always available:
\begin{theorem}\label{thm52} Let $S_n=\sumk X_k$.Then, \\[2mm]
\hspace*{1cm}$\displaystyle\frac{S_n}{n}\dto\begin{cases}\phantom{-(}p-q,&\ttt{with probability}p,\\[2mm]
\phantom{-(((} 0,&\ttt{with probability}r,\\[2mm]
 -(p-q),&\ttt{with probability}q,\end{cases}\ttt{as}\nifi$.\\[2mm]
Moreover,  $E(S_n/n)\to (p-q)^2\ttt{and} \var (S_n/n)\to  (p-q)^2\big(p+q-(p-q)^2\big)\ttt{as} \nifi$.
\end{theorem}
\begin{proof} If $X_1=\pm 1$ we know from above that $E(T_n)=\pm (1+(n-1)(p-q))$, and that $\var(T_n)=n\,\big(p+q-(p-q)^2\big)$. This tells us that,  $\frac{T_n}{n}\pto \pm (p-q)$ as $\nifi$. The case $X_1=0$ is trivial. The conclusion follows.

Moment convergence is immediate, since $|S_n/n|\leq1$ for all $n$.\vsb\end{proof}
\begin{remark}\label{remark51}\emph{(i) Parallelling the corresponding remark in \cite{111} we may interpret the limit such that the random walk at hand, on average, behaves, asymptotically, like a \emph{delayed\/} coin-tossing random walk.}\\[1mm]
\emph{(ii) 
An alternative way of phrasing the conclusion of the theorem is that}
\[F_{S_n/n}(x)\to p\cdot \delta_{p-q}(x)+r\delta_0+q\cdot\delta_{-(p-q)}(x)\ttt{as}\nifi.\]\vsp
\end{remark}
If, on the other hand, we use a \emph{random normalization\/} we obtain the following result:
\begin{theorem}\label{thm53}  Let $S_n=\sumk X_k$.Then, \\[2mm]
\emph{(a)} \hspace*{1cm}$\displaystyle
\frac{S_n-n(p-q)X_1}{\sqrt{n(p+q-(p-q)^2)}} \dto (p+q){\cal N}_{0,1}+r\delta_0
 \ttt{as}\nifi$;\\[2mm]
\emph{(b)} \hspace*{1cm}$\displaystyle
\frac{S_n-n(p-q)X_1}{n} \asto 0 \ttt{as}\nifi$.
\end{theorem}
\begin{proof} For (a) we use the fact that
\beaa \lefteqn{P\Big(\frac{S_n-n(p-q)X_1}{\sqrt{n}}\le x \Big)}\\&=&p\cdot P\Big(\frac{S_n-n(p-q)X_1}{\sqrt{n}}\le x \mid X_1=1\Big)  + q\cdot P\Big(\frac{S_n-n(p-q)X_1}{\sqrt{n}}\leq x \mid X_1=-1\Big) \\
&& \ +\, r\cdot P\Big(\frac{S_n-n(p-q)X_1}{\sqrt{n}} \le x \mid X_1=0\Big)\\
&\dto& (p+q)\cdot {\cal N}_{0,p+q-(p-q)^2}(x)  +r\cdot\delta_{0}(x)
\eeaa
and Proposition \ref{prop51}. 

Conclusion (b) holds true for almost all $\omega \in \{|X_1|=1\}$ by Propostion \ref{prop51} and trivially on $\{X_1=0\}$.\vsb
\end{proof}
\begin{remark}\emph{There is no general LIL in the ERWD, since $X_1=0$ induces the zero process.}
\vsb\end{remark}

\section{Remembering only the distant past 2}\label{anfang2}
\setcounter{equation}{0}
In this section we assume that the elephant only remembers the first two steps, so that ${\cal F}_1=\sigma(X_1),\,
\cfn=\sigma\{X_1, X_2\},\, n \ge 2$. We have to distinguish between the following initial cases:
\begin{itemize}
\item[(a)]\quad $X_1=X_2=1$,  with  $P(X_1=X_2=1)=p^2$;
\item[(b)]\quad $X_1=X_2=-1$, with  $P(X_1=X_2=-1)=pq$;
\item[(c)]\quad $X_1=-X_2=1$, with  $P(X_1=-X_2=1)=pq$;
\item[(d)]\quad $X_1=-X_2=-1$, with  $P(X_1=-X_2=-1)=q^2$;
\item[(e)]\quad $X_1=1, X_2=0$, with  $P(X_1=1, X_2=0)=pr$;
\item[(f)]\quad $X_1=-1, X_2=0$, with  $P(X_1=-1, X_2=0)=qr$;
\item[(g)]\quad $X_1=0$, with  $P(X_1=0)=r$.
\end{itemize}
If we fix the starting values $X_1=x_1$,  $X_2=x_2$ with $x_i \in \{-1,0,1\}$, we get an ERWD with partial sums $T_n=T_n(x_1,x_2), \, n=1,2,\dots $, where $T_2=x_1+x_2$. For $n \ge 3$ the steps are i.i.d.\ summands $\tilde{X}_n$ satisfying $E(\tilde{X}_n)=\big(x_1(p-q) +x_2(p-q)\big)/2=\mu(x_1,x_2)$ and $\var(\tilde{X}_n)= 
(p+q)(x_1^2+x_2^2)/2-\mu^2(x_1,x_2)=\sigma^2(x_1,x_2)\,.$
This yields 
\[E(T_n)= (n-2)\, \mu(x_1,x_2) +x_1+x_2\ttt{ and }\var(T_n)= (n-2)\sigma^2(x_1,x_2).
\]
In particular,
\begin{itemize}
\item[(a)]\quad $\mu(1,1)=p-q\ttt{and}\sigma^2(1,1)=p+q-(p-q)^2$;
\item[(b)]\quad $\mu(-1,-1)=-(p-q)\ttt{and}\sigma^2(-1,-1)=p+q-(p-q)^2$;
\item[(c)]\quad $\mu(1,-1)=0\ttt{and}\sigma^2(1,-1)=p+q$;
\item[(d)]\quad $\mu(-1,1)=0\ttt{and}\sigma^2(-1,1)=p+q$;
\item[(e)]\quad $\mu(1,0)=(p-q)/2\ttt{and}\sigma^2(1,0)=\frac{p+q}{2}-(\frac{p-q}{2})^2$;
\item[(f)]\quad $\mu(-1,0)=-(p-q)/2\ttt{and}\sigma^2(-1,0)=\frac{p+q}{2}-(\frac{p-q}{2})^2$;
\item[(g)]\quad $\mu(0,0)=0\ttt{and}\sigma^2(0,0)=0$.
\end{itemize}
Summarizing we have an ERWD, which in each branch (with given $(x_1,x_2)$) yields a CLT, an  LLN, and an LIL, except,
of course, for case (g). Summarizing we find,  in analogy with the results in Section \ref{anfang1},

\begin{theorem}\label{thm61} Let $S_n=\sumk X_k$. Then \\[2mm]
\emph{(a)} 
\hspace*{1cm}$\displaystyle\frac{S_n}{n}\dto\begin{cases}\phantom{-(}p-q,&\ttt{with probability}p^2,\\
\phantom{-}(p-q)/2,&\ttt{with probability}pr,\\
\phantom{5555}0,&\ttt{with probability}pq+q^2+r,\\
-(p-q)/2,&\ttt{with probability}qr,\\
 -(p-q),&\ttt{with probability}pq,\end{cases}\ttt{as}\nifi$;\\[2.5mm]
\emph{(b)} \hspace*{1cm} $E(S_n/n)\to \dfrac{(p-q)^2}{2}(1+p-q)$\ttt{and}  \\[1mm]
\hspace*{14mm}$\var (S_n/n)\to \dfrac{(p-q)^2}{4}\Big((p+q)(1+3p-q)-(p-q)^2\big(1+(p-q)\big)^2\Big)
\ttt{as}\nifi.$
\end{theorem}
Once again, random normalization produces further limit results:
\begin{theorem}\label{thm62}  Let $S_n=\sumk X_k$.Then, \\[2mm]
\emph{(a)} \hspace*{1cm}$\displaystyle
\frac{S_n-n(p-q)\,(X_1+X_2)/2}{\sqrt{n }} \dto p(p+q)\cdot{\cal N}_{0,p+q-(p-q)^2}\\
\hspace*{6,5cm}+r(p+q)\cdot {\cal N}_{0,(p+q)/2-((p-q)/2)^2}\\[2mm]
\hspace*{6,5cm}+q(p+q)\cdot{\cal N}_{0,p+q}+r\delta_0
 \ttt{as}\nifi$;\\[2mm]
\emph{(b)}
$\hspace*{1cm}\displaystyle
\frac{S_n-n(p-q)\,(X_1+X_2)/2}{n} \asto 0 \ttt{as}\nifi $;\\[2mm]
\emph{(c)} For almost all $\omega \in \{\omega : X_1^2(\omega)+X_2^2(\omega)\not=0\}$,\\[2mm]
$\hspace*{1cm}\displaystyle
\limsup_{n \to \infty}(\liminf_{n \to \infty})\frac{S_n-n(p-q)\,(X_1+X_2)/2}{\sqrt{2n\log \log n}} 
= \sigma(X_1,X_2)\quad (- \sigma(X_1,X_2))$.
\end{theorem}
\proof   The proof of (a) follows by conditioning on $X_1$ and $X_2$, together with the limit theorems for the branches. For further details check the corresponding proof in \cite{111}.
\vsb

We close this section by mentioning the obvious fact that if the elephant remembers the first $m$ random variables for some  $m \in \mathbf{N}$, a further elaboration of our method can be used for additional information.

\section{Remembering only the recent past}\label{ende1}
\setcounter{equation}{0} 
The delayed case differs drastically from the non-delayed case in that, whenever $X_n=0$ for some $n$, then all following summands are also equal to zero, since they emanate from $X_n$, suggesting that for almost all $\omega$ there exists $n_0$, such that $X_n=0$ for all $n>n_0$.  In order to see that this is, indeed, the case we note that 
\[P(X_{n}\neq0)=(p+q)P(X_{n-1}\neq0)= \cdots =(p+q)^{n-1}P(X_1\neq0)=(p+q)^n,\]
from which it follows if $r>0$ that $\sumin P(X_n\neq0)<\infty$, and, hence, by the first Borel--Cantelli lemma that
\begin{equation}\label{finite}P(X_n\neq0\mbox{ i.o.})=0\,.
\end{equation}
Recalling that $r=1-p-q$ we furthermore obtain
\begin{lemma}\label{Lemma7.1} Let $r>0$, and define
\[\tau=\min\{n:X_n=0\}.\] 
Then
\begin{equation}\label{tau}
P(\tau=n)= r(1-r)^{n-1}=(1-p-q)\,(p+q)^{n-1}\quad n\ge 1.  
\end{equation}
Moreover,  $ E(\tau)=\frac{1}{r}=\frac{1}{1-p-q}$.
\end{lemma}
\begin{proof} We have
\[P(\tau=n)=P(X_k\neq0,X_2\neq0, \ldots, X_{n-1}\neq0, X_n=0)=(1-r)^{n-1}r,\]
a geometric distribution whose expected value is as claimed.\vsb
\end{proof}

This, together with (\ref{finite}), yields the following result.
\begin{theorem}\label{Thm7.2}
If $r>0$, then
\[S_n\asto S_\tau\ttt{as}\nifi\ \ttt{and} \  E(S_\tau)=\frac{p-q}{1-p+q}\,.\]
\end{theorem}

\begin{proof} 
We just have to verify the formula for $E(S_\tau)$. Toward that end,
\[
E(S_\tau)=\sum_{n=1}^\infty E(S_n\,|\, \tau=n) P(\tau=n)=\sum_{n=1}^\infty E(S_{n-1}\,|\,X_1\not=0,\dots,X_{n-1}\not=0) \cdot r (1-r)^{n-1}.\]
Now,  $S_{n-1}$ given $X_1\not=0,\dots,X_{n-1}\not=0$ is constituted via an ordinary ERW $\tilde{S}_{n-1}$ with summands $\tilde{X}_k$, obeying transition probabilities $\tilde{p}=p/(p+q)$ for $+1$ and $\tilde{q}=q/(p+q)$ for $-1$. Now, let, for $n\geq1$,  $\tilde{T}_n=\sum_{k=1}^{n} \tilde{X}_k$  with given $\tilde{X_1}=1$.
By the results in \cite{111}, Section 8, we find that
\[E(\tilde{T}_n)=\frac{1-(2\tilde{p}-1)^n}{2(1-\tilde{p})}.\]
If we start with $\tilde{X}_1=-1$ we obtain the negative value of that.
Hence,
\beaa
E(\tilde{S}_{n-1})&=&\frac{p}{p+q}E(\tilde{T}_{n-1})-\frac{q}{p+q}E(\tilde{T}_{n-1})\\
&=& \Big(\frac{p}{p+q}-\frac{q}{p+q}\Big)\, \frac{1-\Big(\frac{2p-(p+q)}{p+q}\Big)^{n-1}}{2\Big(1-\frac{p}{p+q}\Big) }=\frac{p-q}{2q}\Big(1-\Big(\frac{p-q}{p+q}\Big)^{n-1}\Big)\,.
\eeaa
The law of total probability then tells us that
\beaa
E(S_\tau)&=& \sum_{n=1}^\infty E(\tilde{S}_{n-1}) r(1-r)^{n-1}
= \frac{p-q}{2q} \sum_{n=1}^\infty \Big(1-\Big(\frac{p-q}{p+q}\Big)^{n-1}\Big) (1-p-q)(p+q)^{n-1}\\
&=& \frac{p-q}{2q} \Big(1- \frac{1-p-q}{1-(p-q)}\Big)=\frac{p-q}{1-p+q}\,.
\eeaa
\end{proof}
\begin{remark}\emph{ One can also calculate higher moments based on those of $\tilde{T}_n$, but the calculations become very tedious, as all remainder terms in \cite{111} have to be taken into account.}\vsb
\end{remark}
A logical next section would contain analogous elaborations in the case when the elephant remembers the  two most recent steps.  We leave it to the reader(s) to delve further on this matter.

\section{Remembering the distant as well as the recent past}\label{beides}
\setcounter{equation}{0}
We, finally, consider the case when the elephant has a clear memory of the early steps as well as the very recent ones. 
One may imagine a(n old) person who remembers the early childhood and events from the last few days but nothing in between. The most elementary case, which is the one we shall investigate,  is when $\cfn =\sigma\{X_1,X_n\}$ for all $n\geq2$.  As always we begin by assuming that $X_1=1$. Then, for  $n\geq2$,
\[E(X_2)=E\big(E(X_2\mid X_1)\big)=E\big((p-q)X_1\big)=p-q,\]
and
\[
E(\xnp)=E\big(E(\xnp\mid\cfn)\big)=(p-q)E\Big(\frac{X_1+X_n}{2}\Big)=\frac{p-q}{2}\cdot(1+E(X_n)),\]
so that, via Proposition \ref{diff}(ii),
\bea \label{expxnn}
E(X_n)=\frac{p-q}{2+q-p}+\Big(\frac{p-q}{2}\Big)^{n-1}\cdot \frac{2(1+q-p)}{2+q-p}\ttt{for} n\geq1,
\eea 
and, hence, 
\bea\label{exptnn}
E(T_n)= n\cdot\frac{p-q}{2+q-p}+ \frac{4(1+q-p)}{(2+q-p)^2} +o(1)\ttt{as}\nifi\,.
\eea
In order to establish a difference equation for the second moment, we first have to compute the mixed moment. The usual approach, inserting $E(T_n)$ from above (noting that   $E(\xnp^2)=p+q$),  yields
\beaa
E(\tnp\xnp)&=&E\Big(T_{n}(p-q)\frac{1+X_{n}}{2}\Big)+E(\xnp^2)
=\frac{p-q}{2}\big( E(T_{n})+E(T_{n}X_{n})\big) +p+q\\
&=&\frac{p-q}{2}\cdot E(T_{n})+\frac{p-q}{2}\cdot E(T_{n}X_{n})\big) +p+q\\[2mm]
&=& a E(T_{n}X_{n}) +b_n,
\eeaa
with 
\[a=\frac{p-q}{2}\ttt{ and } b_n=n\cdot\frac{(p-q)^2}{2(2+q-p)}+\frac{2(1+q-p)(p-q)}{(2+q-p)^2}+p+q.\] 
Exploiting  Proposition \ref{diff}(i), we then obtain, letting  $\nifi$,
\beaa E(T_nX_n)&=&\Big((n-1) \cdot\frac{(p-q)^2}{2(2+q-p)} +\frac{2(1+q-p)(p-q)}{(2+q-p)^2}+p+q\Big)\Big/
\Big(1-\frac{p-q}{2}\Big)\\
  &&\hskip1pc - \,\frac{p-q}{2n}\Big((n-1) \cdot\frac{(p-q)^2}{2(2+q-p)} 
+\frac{2(1+q-p)(p-q)}{(2+q-p)^2}+p+q\Big)\Big/\Big(1-\frac{p-q}{2}\Big)^2\\[2mm]
&&\hskip2pc \times\big(1+o(1)\big)\\[2mm]
&=&n\cdot \frac{(p-q)^2}{(2+q-p)^2}+\frac{2(p-q)(2+3q-3p)}{(2+p-q)^3}+\frac{2(p+q)}{2+q-p} +o(1).
\eeaa
Next we note (recall (\ref{uli5})) that,  for $n\geq1$,
\beaa
E(T_n^2)&=&E(T_{n-1}^2)+2E(T_{n-1}X_n) +E(X_n^2)\\[2mm]
&=& E(T_{n-1}^2)+2E(T_{n}X_{n})-E(X_n^2)\\[2mm]
&=&E(T_{n-1}^2)+2 \cdot \Big(n\cdot\frac{(p-q)^2}{(2+q-p)^2} +
\frac{2(p-q)(2+3q-3p)}{(2+p-q)^3}+\frac{2(p+q)}{2+q-p}+o(1)\Big)-(p+q),\eeaa
after which telescoping tells us that 
\beaa
 E(T_{n}^2)&=& E(X_1^2)+\frac{n(n+1)}{2}\cdot2 \cdot \frac{(p-q)^2}{(2+q-p)^2}\\[2mm]
&&\hskip2pc +\,n\cdot 2\cdot\Big(\frac{2(p-q)(2+3q-3p)}{(2+p-q)^3}+\frac{2(p+q)}{2+q-p} +o(1)\Big)
-n\cdot(p+q)\\[2mm]
&=&n^2\cdot\frac{(p-q)^2}{(2+q-p)^2} +n\cdot\Big(\frac{(p-q)^2}{(2+q-p)^2}+\frac{4(p-q)(2+3q-3p)}{(2+q-p)^3}+\frac{4(p+q)}{2+q-p}\nonumber\\
&&\hskip2pc -\,  (p+q)\Big)+o(n)\ttt{as}\nifi.
\eeaa
Joining the expressions for the first two moments now tells us that the variance is linear in $n$:
\bea\label{vartnzz}
\var(T_n)&=&n^2\cdot\frac{(p-q)^2}{(2+q-p)^2} 
+n\cdot \Big(\frac{(p-q)^2}{(2+q-p)^2}+\frac{4(p-q)(2+3q-3p)}{(2+q-p)^3}+\frac{4(p+q)}{2+q-p} \nonumber\\[2mm]
&&\hskip2pc-\, (p+q) \Big)  +o(n)-\Big( n\cdot\frac{p-q}{2+q-p}+ \frac{4(1+q-p)}{(2+q-p)^2} \Big)^2+o(n)\nonumber\\[2mm]
&=& n \cdot\Big(p+q+\frac{(p-q)^2}{(2+q-p)^2}+\frac{4(p-q)(2+3q-3p)}{(2+q-p)^3} + \frac{4(p+q)}{2+q-p}-2(p+q)\nonumber\\
&&\hskip2pc -\,2\cdot\frac{p-q}{2+q-p}\cdot \frac{4(1+q-p)}{(2+q-p)^2} \Big)  +o(n)\nonumber\\
&=&n\cdot\Big(p+q  +\frac{p-q}{(2+q-p)^3}\Big((p-q)(-2+q-p)+2(p+q)(2+q-p)^2\Big)\Big)
 +o(n) \nonumber\\[2mm]
&=&n\cdot\sigma_T^2 +o(n)\ttt{as}\nifi\,.
\eea
Given the expressions for mean and variance, a weak law is now immediate:
\bea\label{wllntn10}
\frac{T_n}{n}\pto \frac{p-q}{2+q-p}\ttt{as}\nifi.
\eea
In analogy with our earlier results this suggets that $T_n$ is asymptotically normal. That this is, indeed, the case follows from the fact that $\{T_n,\ n\geq1\}$ is, once again,  a uniformly ergodic Markov chain, 
since $P(X_{n+1}=k\,|\, X_2,\dots,X_n)=P(X_{n+1}=k\,|\, X_n)$ for $k=-1,0,1$ and we have a stationary recurrent and finite state Markov chain. We may thus apply Corollary 5 of \cite{jones} (cf.\ also \cite{ibralinn}, Theorem 19.1) to conclude that $T_n-E(T_n)$ is asymptotically normal with mean zero and variance $\sigma^2_Tn$, with $\sigma^2_T$ as defined in (\ref{vartnzz}). This establishes, in view of (\ref{exptnn}),  that
\bea\label{clt10prel}
\frac{T_n-n\,\frac{p-q}{2+q-p}}{\sigma_T\sqrt{n}}\dto {\cal N}_{0,\,1} \ttt{as}\nifi.
\eea
For $X_1=-1$ and $X_1=0$ we obtain the usual scenarios. This summarizes our findings into 
\beaa
 E(S_n)&=&p E(T_n)-q \,E(-T_n)+0=(p-q)E(T_n),\\
E(S_n^2)&=&(p+q) \,E(T_n^2),\\
\var (S_n)&=&(p+q) E(T_n^2)-(p-q)^2(E(T_n))^2,
\eeaa
and, hence, as $\nifi$,
\bea\label{snlimit}
E(S_n/n)\to\frac{(p-q)^2}{2+q-p}\ttt{and}\var(S_n/n)\to
\sigma_S^2 =\frac{(p-q)^2}{(2+q-p)^2}\big(p+q-(p-q)^2\big)\,.
\eea
In analogy with our earlier results we finally arrive at the following  two limit theorems for $S_n$:
\begin{theorem}\label{snast} We have
\[\frac{S_n}{n}\dto S=\begin{cases}\phantom{-\,}\dfrac{p-q}{2+q-p},&\ttt{with probability} p,\\[1mm]
\hspace*{12mm} 0 ,&\ttt{with probability} r,\\[1mm]
 -\,\dfrac{p-q}{2+q-p},&\ttt{with probability}q,\end{cases}\ttt{as}\nifi.\]
Moreover, $E(S_n/n)^r\to E(S^r)$ for all $r>0$, since $|S_n/n|\leq 1$ for all $n$.
\end{theorem}
\begin{theorem}\label{thmclt10} \qquad$\dfrac{S_n-n\frac{(p-q)X_1}{2+q-p}}{\sqrt{n}}\dto (p+q){\cal N}_{0,\sigma_T^2} +r\delta_0\ttt{as}\nifi$.
\end{theorem}

\section{Final remarks}\label{anm}
\setcounter{equation}{0}

\textbf{(i)}\quad
The extension from \cite{bercu} and \cite{111} consists of allowing for delays at each step--there is a probability $r$ that the elephant stays put at each step. Now, ''conditional'' on the value of $r$ we note that our limit points are functions of $p-q$. Therefore, replacing $q$ by $1-p-r$, we can rewrite our results in terms of functions of $2p-1-r$. In particular, if $r=0$, we observe more clearly how our results reduce to, and extend those of, the two mentioned predecessors.

\noindent\textbf{(ii)}\quad
When the elephant remembers the whole past there is a phase transition at $p=3/4$; see \cite{bercu}. In our setting (Theorem 4.1) the critical point is when $p-q=1/2$. Reinterpreting this as in the previous remark this means that the critical point  is $2p-1-r=1/2$, which reduces to $p=3/4$ when $r=0$. As in \cite{111} there is no phase transition in the case of restricted memories, and, as there, one might ask for the the breaking point. 

\noindent\textbf{(iii)}\quad As was mentioned in \cite{111},  one might think of cases when the memory covers early and/or recent steps, where the length of the memory depends on $n$, typically $\log n$ or $\sqrt{n}$. This may, in addition, have some interest with regard to the previous remark.

\subsection*{Acknowledgement}
We wish to thank a referee for his/her careful scrutiny of the mansucript including the discovery of some obscurity.

\medskip\noindent {\small Allan Gut, Department of Mathematics,
Uppsala University, Box 480, SE-751\,06 Uppsala, Sweden;\\
Email:\quad \texttt{allan.gut@math.uu.se}\\
URL:\quad \texttt{http://www.math.uu.se/\~{}allan}}
\\[4pt]
{\small Ulrich Stadtm\"uller, Ulm University, Department of Number
Theory
and Probability Theory,\\ D-89069 Ulm, Germany;\\
Email:\quad \texttt{ulrich.stadtmueller@uni-ulm.de}\\
URL:\quad
\texttt{http://www.mathematik.uni-ulm.de/en/mawi/institute-of-number-theory-and-probability-\\theory/people/stadtmueller.html}}

\end{document}